\documentclass[12pt]{amsart}

\usepackage[all,color]{xy}
\usepackage{pb-diagram}
\usepackage[mathscr]{eucal}

\usepackage{hyperref}
\hypersetup{
    colorlinks=true, 
    linktoc=all,     
    linkcolor=blue,  
}
\usepackage{amsmath,amsfonts,amsthm,amssymb,amsthm}
\usepackage{tikz,quiver}
\usepackage{xcolor}
\usepackage{float}
\usepackage{pifont}
\usepackage{combelow} 

\RequirePackage{ifthen,setspace,enumitem,tikz}
\usetikzlibrary{arrows}

\DeclareMathAlphabet{\mathpzc}{OT1}{pzc}{m}{it}

\theoremstyle{plain}

\newtheorem{lem}{Lemma}[section]
\newtheorem{cor}{Corollary}[section]
\newtheorem{theorem}{Theorem}[section]
\newtheorem{prop}{Proposition}[section]
\newtheorem{dfn}{Definition}[section]

\newcommand{\pare}[1]{\left( #1 \right)}
\newcommand{\brac}[1]{\left\{ #1 \right\}}
\newcommand{\sqbr}[1]{\left[ #1 \right]}
\newcommand{\norm}[1]{\left\| #1 \right\|}
\newcommand{\inner}[1]{\langle #1 \rangle}
\newcommand{\abs}[1]{\left| #1 \right|}
\newcommand{\set}[2]{\left\{ #1 : #2 \right\}}

\newcommand{\C}{\mathbb{C}}

\newcommand{\Q}{\mathbb{Q}}

\newcommand{\del}{\partial}
\newcommand{\delbar}{\bar{\del}}
\newcommand{\pot}{\sqrt{-1}\del\delbar}
\newcommand{\Pone}{\mathbb{P}^1}
\newcommand{\OPone}{\mathcal{O}_{\Pone}\pare{-1}}
\newcommand{\Kahler}{Kähler }
\newcommand{\MAlow}{\operatorname{MA}_-}
\newcommand{\MAup}{\operatorname{MA}_+}
\renewcommand{\epsilon}{\varepsilon}

\numberwithin{equation}{section}

\hypersetup{pdftitle={Monge-Ampère degenerations and conifold contractions in Fujiki class C},
 pdfauthor={Riley Guyett, and Ryan Mc Gowan}}
\begin{document}

\title[Conifold contractions in Fujiki class $\mathcal{C}$]{Monge-Ampère degenerations and conifold contractions in Fujiki class $\mathcal{C}$}
\author{Riley Guyett and Ryan Mc Gowan}
\raggedbottom
\address{Department of Mathematics, Rutgers University, Piscataway, NJ 08854}
\email{riley.guyett@rutgers.edu}
\address{Department of Mathematics, Rutgers University, Piscataway, NJ 08854}
\email{rmm396@math.rutgers.edu}

\begin{abstract}
Given a $\mathbb{Q}$-Gorenstein normal projective Calabi-Yau variety $Y$ with ordinary double points as singularities and a small crepant resolution $\pi\colon X\to Y$ with $X$ in the Fujiki class $\mathcal{C}$, we establish uniform estimates away from the exceptional locus where the given metrics may degenerate. When $Y$ is a threefold, we prove that these metrics also converge in a Gromov-Hausdorff sense to the metric completion of $Y$ on the regular locus. In particular, our results extend the work of Song on a conjecture of Candelas and de la Ossa.
\end{abstract}
\maketitle 

\section{Introduction}

Yau's solution of the Calabi conjecture \cite{Yau1978} has led to a rich theory of the complex Monge-Amp\`ere equation in the K\"ahler setting, and helped bridge many gaps between complex geometry and string theory. For instance, Reid's fantasy, originally posed by Reid in \cite{Re} and building upon the work of \cite{clemens, Fr, hirzebruch,T1}, conjectures that all Calabi-Yau threefolds are connected via some geometric transitions in some bigger moduli space. The work of Strominger \cite{Strominger} highlighted the usefulness of a specific geometric transition, conifold transitions, in a physical context. Later, Candelas and de la Ossa \cite{CaO} conjectured that the algebraic model posed by Reid should also have an analytic and geometric interpretation, and Candelas, Green and H\"ubsch described such transitions in \cite{CGH}. Rong and Zhang \cite{RongZhang2011} and later Song \cite{Song} proved that Calabi-Yau threefolds can be connected by a continuous path, in a Gromov-Hausdorff sense, of Ricci-flat Calabi-Yau metrics via conifold transitions.

Moving to the hermitian situation introduces many subtleties not present in the K\"ahler case, since one now must deal with both torsion of the metrics and defining the Monge-Amp\`ere mass for a nonclosed form. For the former, Tosatti and Weinkove \cite{TW-AJM, TW-JAMS} extended Yau's results by studying Chern-Ricci-flat metrics, rather than the Ricci-flat metrics Yau originally produced in \cite{Yau1978}, and furthermore obtained estimates for the hermitian complex Monge-Amp\`ere equation on compact hermitian manifolds. Tosatti also addressed how to handle degenerating metrics in \cite{Tosatti2009}. For the latter, pluripotential methods pioneered by Bedford and Taylor \cite{BT} and expanded upon by Ko{\l}odziej \cite{kolodziej,kolodziej2} have proven to be powerful tools in complex geometry, for instance, the works \cite{B, BEGZ, Demailly,EGZ, PSS,PS,PS2}. Such methods have been employed as in \cite{CT,GL22,GL23} to bound the Monge-Amp\`ere masses for such degenerate equations. More recently, Boucksom, Guedj, and Lu \cite{BGL} extended the theory of non-pluripolar products to possibly nonclosed currents on compact Hermitian manifolds with the bounded mass property, and established existence results for degenerate Monge-Amp\'ere equations in big classes. Their framework provides a global interpretation of the limiting equation arising from the Hermitian degenerations considered here. In this paper, we synthesize these methodologies described above to understand conifold transitions when the resolution $X$ is in Fujiki class $\mathcal{C}$, and hence need not be K\"ahler. 

We now state the main results of the paper. Let $Y$ be a normal irreducible projective $\Q$-Gorenstein variety of complex dimension $n$ with torsion canonical divisor and let $\pi\colon X\to Y$ be a smooth small crepant resolution and $E=\mathrm{Exc}(\pi)$. By crepancy, we have a smooth positive form $\Omega$ on $X$ normalized by the conditions
\[
\pot\log\Omega=0, \qquad \int_X \Omega=1,
\]
and define $\Omega_Y=\pi_*\Omega$. Fix a smooth hermitian form $\omega_Y$ on $Y$, and set $\theta = \pi^*\omega_Y$. In particular, $\theta$ is a smooth and semipositive $\pare{1,1}$-form on $X$, which is strictly positive on $X\setminus E$. For every $0<\varepsilon\leq 1$ and given a hermitian metric $\eta$ on $X$, by the work of \cite{TW-JAMS} there exists $\varphi_\varepsilon$ and $c_\varepsilon$ such that
\begin{equation}\label{cma}
    \tilde{\omega}_\varepsilon=\theta+\varepsilon\eta+\pot \varphi_\varepsilon>0, \qquad \tilde{\omega}_\varepsilon ^n =c_\varepsilon\Omega, \qquad \sup_X \varphi_\varepsilon=0 
\end{equation}
such that $\tilde{\omega}_\varepsilon$ is Chern-Ricci-flat. Our first result is the following: 
\begin{theorem}\label{MainTheorem1}
    Let $\pi:X\to Y$ be a small smooth crepant resolution. For $0<\varepsilon\leq 1$, solutions $(\varphi_\varepsilon,c_\varepsilon)$ to Equation \ref{cma} satisfy $c_\epsilon\searrow c_0$ and the uniform estimate
    \[
    \sup_\varepsilon\|\varphi_\varepsilon\|_{L^\infty(X)}<\infty.
    \]
    On a compact set $K\Subset X\setminus E$, for $k\geq 0$, we have higher-order estimates
    \[
    \sup_\varepsilon\|\varphi_\varepsilon\|_{C^k(K,\eta)}<\infty.
    \]
    For every sequence $\varepsilon_j\searrow0$, there is a subsequence along which the potentials $\varphi_{\varepsilon_j}$ converge in $L^1(X)\cap C_{loc}^\infty(X\setminus E)$ to a bounded solution $\varphi_0\in \mathrm{PSH}(X,\theta)$ of
    \[
    \langle(\theta+\pot\varphi_0)^n \rangle=c_0\Omega,\qquad \sup_X \varphi_0=0,
    \]
    where the angle brackets denote the nonclosed non-pluripolar product of Boucksom-Guedj-Lu \cite{BGL}. 
\end{theorem}
We also have the following special case:
\begin{cor}\label{KahlerMainTheorem1}
    When $\omega_Y$ is K\"ahler, the constant in Theorem \ref{MainTheorem1} obeys $c_0=\mathrm{Vol}(\theta)$, and $\varphi_\varepsilon$ converges to $\pi^*\psi_Y$, where $\psi_Y$ is the unique bounded solution to 
    \[
    (\omega_Y+\pot\psi_Y)^n =\left(\int_Y\omega_Y^n\right)\Omega,\qquad \sup_Y \psi_Y=0.
    \]
\end{cor}
We stress that in the nonclosed setting, the uniqueness of the bounded potential is not known at this moment to the authors. We only claim the full-family convergence for the K\"ahler case.

We now specialize to the case when $Y$ is a threefold with finitely many ordinary double point singularities modeled by a conifold singularity. Conifold singularities have garnered interest in their own right, for instance see the works of \cite{DonaldsonSunOne,DonaldsonSunTwo,Fr,HeinSun2017,tianconifol,Zhang2025}. Assuming that $\omega_Y$ is K\"ahler and $\eta$ is chosen to be K\"ahler near the exceptional locus $E$, we can define a distance associated to the metric $\omega_Y+i\partial\bar\partial\psi_Y$ on $Y_{\mathrm{reg}}$. We may apply Theorem \ref{MainTheorem1} to establish our second main result and an extension of the work of Song in \cite{Song1.5,Song2,Song}:
\begin{theorem}\label{MainTheorem2}
    Let $\pi\colon X\to Y$ be smooth small resolution, with $\dim_\C Y=3$ and $Y$ having only ordinary double points. Further assume that $\omega_Y$ is Kähler and that $\eta$ is Kähler on a neighborhood of $E$. If $\tilde{\omega}_\epsilon$ is as in Equation \ref{cma}, then: 
    \begin{enumerate}
        \item The metric completion $\overline{\pare{Y_{\text{reg}},d_Y}}$ is canonically homeomorphic to $Y$.
        \item If $g_\epsilon$ denote the Riemannian metrics associated with $\tilde{\omega}_\epsilon$ for $0<\epsilon\le 1$, then
        \begin{align*}
        \pare{X,g_\epsilon} &\stackrel{GH}{\longrightarrow} \overline{\pare{Y_{\text{reg}},d_Y}}.
    \end{align*}
    \end{enumerate}
\end{theorem}
Understanding these conifold transitions in the non-K\"ahler setting has received attention recently; we refer the reader to \cite{CPY,FPS,GiustiSpotti,Katz}.\bigskip

The organization of this article is as follows. Section 2 establishes the preliminary background material, primarily concerning lower and upper Monge-Ampère masses. Section 3 proves uniform and torsion-sensitive local zeroth- and second-order estimates and leverages these estimates to pass to the global non-pluripolar limit. In the K\"ahler case, we also deal with descent and uniqueness of the potentials. In Section 4, we specialize to a nodal threefold and derive similar estimates to Song. Finally, in Section 5, we identify the metric completion under an easily verified ``no-neck hypothesis'' and prove Gromov-Hausdorff convergence via an abstract comparison argument.

\bigskip
\noindent\textbf{Acknowledgments.}
The authors thank their advisor, Professor Jian Song, for introducing them to this problem and for his encouragement. This work forms part of the authors' Ph.D. theses at Rutgers University, and they are grateful to the Department of Mathematics for its generous support. AI tools were used to check routine steps in the proofs and to improve the exposition.

\section{Preliminaries}

Let $X$ be a compact complex manifold and let $Y$ be a $\Q$-Gorenstein normal projective variety, both of complex dimension $n$. By $\Q$-Gorenstein we mean that $mK_Y$ is Cartier for some integer $m>0$. We say that a resolution $\pi\colon X\to Y$ is:
\begin{itemize}
    \item A crepant resolution if $K_X=\pi^*K_Y$ as $\Q$-Cartier divisors. 
    \item A small resolution if the exceptional locus $E=\mathrm{Exc}(\pi)$ has codimension at least two.
\end{itemize} 
If $Y$ is a complex threefold, then a small resolution of an ordinary double point at a singularity $p$  replaces the point $p$ with a complex curve $C\cong \mathbb{P}^1$. Locally, the singularity is analytically isomorphic to a quadric affine cone 
\begin{align*}
    \mathfrak{C} = \set{(z_1,z_2,z_3,z_4)\in \C^4}{z_1^2+z_2^2+z_3^2+z_4^2=0}
\end{align*}
while the resolution is locally modeled by
\begin{equation}\label{NodalLocalModel}
    \hat{E} = \mathrm{Tot}(\OPone\oplus \OPone).
\end{equation}
Because $Y$ is projective, after taking a projective resolution $Z\to Y$, we can see that $X$ is bimeromorphic to a compact Kähler manifold and thus $X$ belongs to the Fujiki class $\mathcal{C}$.
\[\begin{tikzcd}
	X && Y \\
	& Z
	\arrow["\pi", from=1-1, to=1-3]
	\arrow["\sim"', dashed, from=1-1, to=2-2]
	\arrow[from=2-2, to=1-3]
\end{tikzcd}\]

The following lemma provides a hermitian metric on $X$ that is Kähler near $E$.

\begin{lem}\label{EtaKahlerNearE}
    In the nodal threefold setup, there exists a hermitian metric $\eta$ on $X$ which is \Kahler in a neighborhood of $E$.
\end{lem}

\begin{proof}
    Choose pairwise disjoint neighborhoods $U_j$ of the exceptional curves. By the local model (Equation \ref{NodalLocalModel}), each $U_j$ carries a \Kahler form $\omega_j$. Let $\chi_j$ be a smooth cutoff function supported in $U_j$ with $0\leq \chi_j\leq 1$ and $\chi_j\equiv 1$ on a small neighborhood $U_j^\prime\Subset U_j$. If $h$ is any global hermitian metric on $X$, we can define
    \begin{align*}
        \eta &= \begin{cases}
            \pare{1-\chi_j}h + \chi_j\omega_j, & \text{if }z\in U_j, \\
            h, & \text{otherwise}.
        \end{cases}
    \end{align*}
    This metric is a pointwise convex sum of smooth hermitian metrics and thus remains smooth and hermitian on $X$. Moreover, $\eta=\omega_j$ on $U'_j$ and hence $\eta$ is \Kahler near $E$ as claimed.
\end{proof}

We wish to solve a hermitian complex Monge-Amp\`ere equation. However, we require some notions not present in the \Kahler setting. We define the lower Monge-Amp\`ere mass and upper Monge-Amp\`ere mass of a smooth semi-positive $(1,1)$-form $\alpha$, respectively, to be 
\begin{align*}
    \operatorname{MA}_-\pare{\alpha} &= \inf_{u\in \operatorname{PSH}\pare{X,\alpha}\cap L^\infty}\int_X\pare{\alpha+\pot u}^n \\
    \operatorname{MA}_+\pare{\alpha} &= \sup_{u\in \operatorname{PSH}\pare{X,\alpha}\cap L^\infty}\int_X\pare{\alpha+\pot u}^n.
\end{align*}
Note that if the form $\alpha$ is not closed, neither quantity needs to equal $\int_X\alpha^n$. Additionally, the condition $\int_X \alpha^n>0$ alone does not imply $\operatorname{MA}_-\pare{\alpha}>0$.

For the reader's convenience, we collect some results needed for our arguments. To begin, we will make use of the following results from Guedj-Lu \cite{GL22}:
\begin{prop}\label{MAMonotone}\leavevmode
    \begin{enumerate}
        \item (Proposition 3.2) Let $0\leq \alpha_1\le\alpha_2$ be semi-positive $(1,1)$-forms. Then 
            \begin{align*}
                \operatorname{MA}_\pm\pare{\alpha_1} \le \operatorname{MA}_\pm\pare{\alpha_2}.
            \end{align*}
        \item  (Corollary 3.8) If $X$ belongs to the Fujiki class $\mathcal{C}$, then every hermitian metric on $X$ has finite upper Monge-Ampère mass.
    \end{enumerate}
\end{prop}

We also have the following result of Tosatti-Weinkove \cite{TW-JAMS}, which we rephrase for our setting.

\begin{prop}\label{MAPair}
    For every $0<\epsilon\le 1$, there is a unique constant $c_\epsilon>0$ and a unique smooth function $\varphi_\epsilon$ which satisfy Equation \ref{cma}. Furthermore, $\tilde{\omega}_\epsilon$ is Chern-Ricci-flat.
\end{prop}
\begin{proof}
    The form $\theta_\epsilon\ge \epsilon \eta$ is a hermitian metric, and taking $F_\epsilon=\log \frac{\Omega}{\theta_\epsilon^n}$, Corollary 1 in \cite{TW-JAMS} gives a unique $b_\epsilon$ and $\varphi_\epsilon$ such that $\sup_X \varphi_\epsilon=0$ and 
    \begin{align*}
        \pare{\theta_\epsilon+\pot \varphi_\epsilon}^n &= e^{F_\epsilon+b_\epsilon}\theta_\epsilon^n.
    \end{align*}
    Taking $c_\epsilon=e^{b_\epsilon}$, we recover Equation \ref{cma}. We also have 
    \begin{align*}
        \operatorname{Ric}^{\operatorname{Ch}}\pare{\tilde{\omega}_\epsilon} &= -\pot\log c_\epsilon \Omega = -\pot\log \Omega = 0.\qedhere
    \end{align*}
\end{proof}

We now get a crude bound for the constants $c_\epsilon$.

\begin{prop}\label{Boundedcepsilon}
    There exists $C>0$ such that 
    \[\MAlow\pare{\theta}\le c_\epsilon \le C\]
    for all $0<\epsilon\le 1$. If $\theta$ is closed, we can identify $\MAlow\pare{\theta}$ with $\operatorname{vol}\pare{\theta}$.
\end{prop}
\begin{proof}
    By Proposition \ref{MAMonotone}.1, since $0\le \theta\le \theta_\epsilon \le \theta_1$, monotonicity implies
    \begin{align*}
        \MAlow\pare{\theta} \le  \MAlow\pare{\theta_\epsilon} \le  \MAup\pare{\theta_\epsilon} \le  \MAup\pare{\theta_1}<\infty.
    \end{align*}
    Further, by Equation \ref{cma},
    \begin{align*}
        c_\epsilon &= \int_X\pare{\theta_\epsilon+\pot\varphi_\epsilon}^n
    \end{align*}
    and because $\varphi_\epsilon$ is admissible for the lower and upper Monge-Ampère mass,
    \begin{align*}
        \MAlow\pare{\theta_\epsilon} &\le c_\epsilon \le \MAup\pare{\theta_\epsilon}\le \MAup\pare{\theta_1} = C.
    \end{align*}
    In the case that $\theta$ is closed, Bedford-Taylor integration by parts yields 
    \begin{equation*}
        \MAlow\pare{\theta} = \MAup\pare{\theta} = \int_X \theta^n = \operatorname{vol}\pare{\theta}.\qedhere
    \end{equation*}
\end{proof}

Finally, we require the following reformulation of Theorem 1.1 and Theorem 2.2 of Collins and Tosatti \cite{CT}:

\begin{prop}\label{NullalphaisE}
    Let $X$ be a compact complex manifold in the Fujiki class $\mathcal{C}$ and let $\alpha$ be a closed smooth real $(1,1)$-form whose class $[\alpha]\in H^{1,1}(X,\mathbb{R})$ is nef with $\int_X \alpha^n >0$. Then 
    \[
    E_{nK}(\alpha)=\mathrm{Null}(\alpha).
    \]
    Moreover, there exists a \Kahler current $T$ in the class $[\alpha]$, with analytic singularities such that 
    \[
    E_+(T)=E_{nK}(\alpha).
    \]
    In particular, $E_{nK}(\alpha)$ is an analytic subvariety of $X$.
\end{prop}

\begin{lem}\label{MAposandrhoexists}
Let $Y$ be a compact projective normal variety, let $\omega_Y$ be a smooth hermitian form on $Y$, and let $\pi\colon X\to Y$ be a resolution. Set $\theta=\pi^*\omega_Y$. Then
\begin{align*}
    \MAlow\pare{\theta} >0.
\end{align*}
More precisely, if $\kappa_Y$ is any projective Kähler form and $a>0$ satisfies $\omega_Y\ge a\kappa_Y$, then
\begin{align*}
    \MAlow\pare{\theta} \ge a^n\int_Y\kappa_Y^n>0.
\end{align*}
Moreover, for every hermitian metric $\eta$ on $X$, there exists $\delta>0$ and $\rho\in\operatorname{PSH}\pare{X,\theta}$ with analytic singularities precisely along  $\operatorname{Exc}\pare{\pi}$ such that 
\begin{equation}\label{thetaboundedbelow}
    \theta+\pot\rho \ge \delta\eta.
\end{equation}
\end{lem}
\begin{proof}
    Let $\kappa_Y$ be as above. Since $\omega_Y$ and $\kappa_Y$ are both smooth hermitian forms and $Y$ is compact, there is a constant $a> 0$ such that
    \begin{align*}
        \omega_Y\ge a\kappa_Y.
    \end{align*}
    Pulling back to $X$ we get 
    \begin{align*}
        \theta \ge a\pi^*\kappa_Y =: \alpha.
    \end{align*}
    Observe that $\alpha$ is smooth, closed, big, nef, semipositive, and strictly positive on $X\setminus E$.
    By monotonicity and the closedness of $\alpha$,
    \begin{align*}
        \MAlow\pare{\theta} \ge \MAlow\pare{\alpha}=\int_X\alpha^n>0.
    \end{align*}
    Since $\operatorname{Null}\pare{\alpha}=E$, Proposition \ref{NullalphaisE} yields a Kähler current
    \begin{align*}
        \alpha +\pot\rho &\ge \delta \eta
    \end{align*}
    with analytic singularities precisely along $E$. Since $\theta\ge \alpha$, the same $\rho$ satisfies Equation \ref{thetaboundedbelow}.
\end{proof}

\section{Estimates}

The goal of this section is to derive $C^0$ and $C^2$ estimates for $\varphi_\epsilon$ to obtain  uniform estimates on the metric $\tilde{\omega}_\varepsilon$.  

\begin{prop}\label{ZerothOrder}
Let $\varphi_\epsilon$ be as in Proposition \ref{MAPair}. Then there exists a uniform constant $C>0$ independent of $\varepsilon$ such that
\begin{align*}
    -C \le \varphi_\epsilon \le 0.
\end{align*}
\end{prop}
\begin{proof}
Uniform Skoda integrability
for $\theta_1$ gives constants $a,A>0$ such that
\begin{align*}
    \int_X e^{-au}\Omega\le A
\end{align*}
for all $u\in \operatorname{PSH}\pare{X,\theta_1}$ with $\sup_X u=0$. Thus, for any $m>n$ we see
\begin{align*}
    \sup\left\{\int_X\pare{-u}^m\Omega: u\in \operatorname{PSH}\pare{X,\theta_1}, \sup_X u=0\right\} <\infty.
\end{align*}
If we define 
\begin{align*}
    A_m\pare{\Omega;\alpha} = \sup\set{\pare{\int_X\pare{-u}^m\Omega}^{1/m}}{u\in \operatorname{PSH}\pare{X,\alpha}, \sup_X u=0},
\end{align*}
then the inclusion $\mathrm{PSH}(X, \theta_\varepsilon)\subseteq \mathrm{PSH}(X,\theta_1)$ implies 
\[A_m\pare{\Omega;\theta_\epsilon}\leq A_m(\Omega; \theta_1)<\infty\] 
for all $0<\epsilon\le 1$. Further, Proposition \ref{Boundedcepsilon} gives
\begin{align*}
    \frac{c_\epsilon}{\MAlow\pare{\theta_\epsilon}} &\le \frac{\MAup\pare{\theta_1}}{\MAlow\pare{\theta}}.
\end{align*}
Lemma \ref{MAposandrhoexists} ensures $\textrm{MA}_-(\theta_\varepsilon)>0$, so applying Guedj-Lu {\cite[Theorem~A]{GL23}}, there exists some constant $C>0$ independent of $\varepsilon$ such that 
\begin{equation*}
    \operatorname{Osc}_X(\varphi_\epsilon) \le C.\qedhere
\end{equation*}
\end{proof}

Away from the exceptional locus, we also have the following $C^2$ estimates. 

\begin{prop}\label{OmegaEpsilonBounded}
For every $K\Subset X\setminus E$ there exists $C_K>1$ independent of $0<\epsilon\le 1$ such that on $K$
\begin{align*}
    C_K^{-1}\eta \le \tilde{\omega}_\epsilon \le C_K\eta.
\end{align*}
\end{prop}
\begin{proof}
    We follow a similar argument to Guedj-Lu in {\cite[Theorem 4.2]{GL23}}. 
    
\textit{Step 1. (Local calculations)} By Lemma \ref{MAposandrhoexists}, there exists $\rho\in \mathrm{PSH}(X,\theta)$ with analytic singularities in $E$, smooth on $X\setminus E$, and a constant $\delta>0$ such that 
\begin{align*}
    \theta +\pot \rho &\ge \delta \eta \\
    \sup_X\rho &= 0
\end{align*}
in the sense of currents. We start locally and write
\begin{align*}
    \eta &= \sqrt{-1}h_{i\bar{j}}dz^i\wedge d\bar{z}^j, \\
    \theta_\epsilon &= \sqrt{-1}b_{i\bar{j}}^\epsilon dz^i\wedge d\bar{z}^j, \\
    \tilde{\omega}_\epsilon &= \sqrt{-1}\tilde{g}_{i\bar{j}} dz^i\wedge d\bar{z}^j, \quad\quad\tilde{g}_{i\bar{j}} = b_{i\bar{j}}^\epsilon +\pare{\varphi_\epsilon}_{i\bar{j}}.
\end{align*}
We can write the torsion locally as follows:
\begin{align*}
    T^\epsilon_{ij\bar{k}} = \tilde{g}_{j\bar{k},i}-\tilde{g}_{i\bar{k},j}= \theta_{j\bar{k},i}-\theta_{i\bar{k},j}+\epsilon\pare{h_{j\bar{k},i}-h_{i\bar{k},j}}
\end{align*}
Since $\theta$ and $\eta$ are both smooth forms on a compact manifold, and hence also have smooth components, we get the following bound on the torsion
\begin{equation}\label{TorsionBounded}
    \abs{T_{ij\bar{k}}^\epsilon}\le C,
\end{equation}
where $C$ is a constant for the lowered coordinate components on the chosen finite coordinate cover.

\textit{Step 2. (Yau-style maximum principle argument)} For a smooth positive function $f$ we shall write $\tilde{\Delta}f=\tilde{g}^{i\bar{j}}f_{i\bar{j}}$. Let
\[
S=\operatorname{Tr}_\eta(\tilde\omega_\varepsilon),\qquad R=\operatorname{Tr}_{\tilde{\omega}_\varepsilon}(\eta).
\]
We wish to apply a maximum principle argument to $\tilde\Delta\log(S)$. Let $C$ be the constant from Proposition \ref{ZerothOrder} and define
\begin{align*}
    u = \varphi_\epsilon -\rho + C+2.
\end{align*}
Note that $u>1$ by construction, and $\rho\to-\infty$ along its singular locus.
To deal with the torsion terms present in the hermitian setting, it is useful to consider the quantity
\[
H=\log(S)-\gamma(u)
\]
where $\gamma(u)=Au+\log u$ for some constant $A$ to be determined later. Hence, $H\to -\infty$ along the singular locus, meaning that $H$ attains its maximum value at some point $x_0\in X\setminus E$.
The following formula will prove helpful for our computations:
\begin{equation}\label{laplacian}
    \tilde\Delta\log f = \frac{\tilde\Delta f}{f} - \frac{|\partial f|^2_{\tilde\omega_\varepsilon}}{f^2}.  
\end{equation}
At $x_0$, we can choose holomorphic coordinates of Guan-Li as described in \cite{GuanLi} so that 
\begin{align*}
    h_{i\bar{j}} = \delta_{ij}, \quad h_{i\bar{i},j}=0, \quad \tilde{g}_{i\bar{j}} = \lambda_{i}\delta_{ij}.
\end{align*}
Observe that $S$ and $R$ can be written as
\begin{align*}
    S = \sum_{i=1}^n \lambda_i, \quad  R = \sum_{i=1}^n \lambda_i^{-1}.
\end{align*}
Likewise, we can define
\begin{align*}
    D=\frac{\tilde{\omega}_\varepsilon^n}{\eta^n} = \prod_{i=1}^n \lambda_i.
\end{align*}
Additionally, combined with Proposition \ref{Boundedcepsilon}, there are fixed uniform constants $d_1,d_2$ such that 
\begin{equation}\label{DBounds}
    0 < d_1 \le D
    \le d_2<\infty.
\end{equation}
The goal is now to obtain a uniform upper bound on $H$. Given this upper bound, for fixed $K\Subset X\setminus E$, $\rho$ is bounded below on $K$ and $\varphi_\epsilon$ is uniformly bounded, so $u\le C_K^\prime$ on $K$. Hence, $S\le C_K$ for some $C_K>1$, which means for each $i$,
\begin{align*}
    \lambda_i = \frac{D}{\prod_{j\ne i}\lambda _j} \ge \frac{d_1}{C_K^{n-1}} >0.
\end{align*}
After possibly redefining $C_K$, we have the desired inequality.
To prove this uniform bound, since $S=h^{k\bar{\ell}}\tilde{g}_{k\bar\ell}$ we have
\begin{align}\label{threeterm}
    \tilde{\Delta} S
    &= \sum_{i,k} \lambda_i^{-1}\tilde{g}_{k\bar{k},i\bar{i}}-2\operatorname{Re}\left(\sum_{i,j,k}\lambda_i^{-1}h_{j\bar{k},\bar{i}}\tilde{g}_{k\bar{j},i}\right)+\sum_{i,k,\ell}\lambda_i^{-1}\pare{h^{k\bar{\ell}}}_{i\bar{i}}\tilde{g}_{k\bar{\ell}}\\&=(I)+(II)+(III).
\end{align}
We now estimate each of $(I),(II)$ and $(III)$. Starting with $(III)$, by the Guan-Li coordinates,  
\begin{align*}
    (III)= \sum_{i,k}\lambda_i^{-1}\pare{h^{k\bar{k}}}_{i\bar{i}}\lambda_k.
\end{align*} 
There is a uniform constant $C$ such that $\abs{\pare{h^{k\bar{k}}}_{i\bar{i}}}\le C$, hence 
\begin{align*}
    \abs{(III)} &\le \sum_{i,k}\lambda_i^{-1}\abs{\pare{h^{k\bar{k}}}_{i\bar{i}}}\lambda_k \le C\sum_{i,k}\lambda_i^{-1}\lambda_k = CRS.
\end{align*}
As for $(I)$, the Chern-Ricci identity at the point $x_0$ implies
\begin{align*}
    \operatorname{Tr}_\eta\left( \operatorname{Ric}^{\operatorname{Ch}}\pare{\tilde{\omega}_\epsilon}\right) = \sum_{i,k}\lambda_i^{-1}\pare{-\tilde{g}_{i\bar{i},k\bar{k}}+\sum_j\lambda_j^{-1}\tilde{g}_{i\bar{j},k}\tilde{g}_{j\bar{i},\bar{k}}}.
\end{align*}
Since $\tilde{\omega}_\epsilon$ is Chern-Ricci-flat, the left side vanishes. Additionally, 
\begin{align*}
    \tilde{g}_{k\bar{k},i\bar{i}} &= \tilde{g}_{i\bar{i}, k\bar{k}}+\pare{b^\epsilon_{k\bar{k},i\bar{i}}-b^\epsilon_{i\bar{i},k\bar{k}}}.
\end{align*}
There is a uniform bound $C$ on the bracketed term, so we find
\begin{equation}\label{estimateeye}
    (I)\geq \sum_{i,j,k} \lambda_i^{-1}\lambda_j^{-1}\abs{\tilde{g}_{i\bar{j},k}}^2-CR.
\end{equation}
Finally, for $(II)$ we exploit the orthogonality condition ($h_{i\bar{i}, j}=0$) of the Guan-Li holomorphic coordinates to eliminate any diagonal terms where $j=k$. By the definition of the torsion $T^\varepsilon_{i\bar kj}$, we see
\begin{align*}
    \abs{(II)} \le 2\sum_{i,j\ne k}\abs{\lambda_i^{-1}h_{j\bar{k},\bar{i}}\tilde{g}_{i\bar{j},k}} + 2\sum_{i,j\ne k}\abs{\lambda_i^{-1}h_{j\bar{k},\bar{i}}T^\epsilon_{ik\bar{j}}}.
\end{align*}
Applying Cauchy-Schwarz to the first sum, we find
\begin{align*}
    |(II)|\le \sum_{i,j\ne k}\lambda_i^{-1}\lambda_j^{-1}\abs{\tilde{g}_{i\bar{j},k}}^2 + \sum_{i,j\ne k}\lambda_i^{-1}\lambda_j\abs{h_{j\bar{k},\bar{i}}}^2 + 2\sum_{i,j\ne k}\abs{\lambda_i^{-1}h_{j\bar{k},\bar{i}}T^\epsilon_{ik\bar{j}}}
    \end{align*}
Then, applying Equation \ref{TorsionBounded} we get
\begin{align*}
   |(II)| \le \sum_{i,j\ne k}\lambda_i^{-1}\lambda_j^{-1}\abs{\tilde{g}_{i\bar{j},k}}^2 + CRS + CR.
\end{align*}
Returning to Equation \ref{threeterm}, we now have the following lower bound:
\begin{equation}\label{DeltaS}
    \tilde{\Delta} S \ge \sum_{i,j} \lambda_i^{-1}\lambda_j^{-1}\abs{\tilde{g}_{i\bar{j},j}}^2-CRS-CR.
\end{equation}
We now move to the second ingredient we need to apply the formula in Equation \ref{laplacian}. Taking one derivative, we see by the definition of the torsion
\begin{align*}
    S_i = \sum_j \tilde{g}_{j\bar{j},i} = \sum_j \pare{\tilde{g}_{i\bar{j},j}+T^\epsilon_{ij\bar{j}}}.
\end{align*}
Thus, we see that
\[
\abs{\del S}_{\tilde{\omega}_\varepsilon}^2 = \sum_i \lambda_i^{-1}\abs{\sum_{j}\tilde{g}_{i\bar{j},j}}^2  + \sum_i \lambda_i^{-1}\abs{\sum_j T^\epsilon_{ij\bar{j}}}^2  + 2\operatorname{Re}\left(\sum_{i,j,k} \lambda_i^{-1}T^\epsilon_{ij\bar{j}}\tilde{g}_{k\bar{i},\bar{k}}\right).
\]
Applying Cauchy-Schwarz to the first sum and the bound in Equation \ref{TorsionBounded},  
\begin{equation}\label{gradSi}
\abs{\del S}_{\tilde{\omega}_\varepsilon}^2 \leq S\sum_{i,j} \lambda_i^{-1}\lambda_j^{-1}|\tilde g_{i\bar j,j}|^2 +CR + 2 \operatorname{Re}\left(\sum_{i,j,k} \lambda_i^{-1}T^\epsilon_{ij\bar{j}}\tilde{g}_{k\bar{i},\bar{k}}\right).    
\end{equation}
We now bound the last term. Note that
\begin{align*}
    H_{\bar i} = \frac{S_{\bar i}}{S} - \gamma^\prime\pare{u}\cdot u_{\bar{i}}.
\end{align*}
Since $x_0$ is chosen to be the maximum of $H$, the above vanishes at $x_0$. Hence,
\begin{align*}
   \sum_k \tilde{g}_{k\bar{i},\bar{k}} = S_{\bar{i}}-\sum_k \overline{T^\epsilon_{ik\bar{k}}}=S\gamma^\prime(u)u_{\bar{i}}-\sum_k \overline{T^\epsilon_{ik\bar{k}}}.
\end{align*}
Thus, 
\begin{align*}
&\abs{2\operatorname{Re}\left(
    \sum_{i,j,k}\lambda_i^{-1}
    T^\epsilon_{ij\bar{j}}\tilde{g}_{k\bar{i},\bar{k}}
\right)} \\
&\qquad=
\left|
    2S\gamma^\prime(u)\operatorname{Re}\left(
        \sum_{i,j}\lambda_i^{-1}
        T^\epsilon_{ij\bar{j}}u_{\bar{i}}
    \right)-2\operatorname{Re}\left(
        \sum_i\lambda_i^{-1}
        \abs{\sum_j T^\epsilon_{ij\bar{j}}}^2
    \right)
\right|.
\end{align*}
Once more applying Equation \ref{TorsionBounded}, we get
\[
\abs{2\operatorname{Re}\left(\sum_{i,j,k} \lambda_i^{-1}T^\epsilon_{ij\bar{j}}\tilde{g}_{k\bar{i},\bar{k}}\right)} \le \abs{2S\gamma^\prime(u)\operatorname{Re}\left(\sum_{i,j} \lambda_i^{-1}T^\epsilon_{ij\bar{j}}u_{\bar{i}}\right)}+CR.
\]
Using Young's inequality with weight $-\gamma^{\prime\prime}(u)$,  
\begin{align*}
\abs{2S\gamma^\prime(u)\operatorname{Re}\left(\sum_{i,j} \lambda_i^{-1}T^\epsilon_{ij\bar{j}}u_{\bar{i}}\right)} &\leq 
-S^2\gamma''(u)\sum_i \lambda_i^{-1}\abs{u_i}^2-\frac{\gamma^\prime\pare{u}^2}{\gamma^{\prime\prime}\pare{u}}\sum_i\lambda_i^{-1}\abs{\sum_j T^\epsilon_{ij\bar{j}}}^2\\
& \leq  -S^2\gamma''(u)\abs{\partial u}^2_{\tilde{\omega}_{\varepsilon}}-\frac{\gamma^\prime\pare{u}^2}{\gamma^{\prime\prime}\pare{u}}\cdot CR
\end{align*}
Therefore, applying Equation \ref{laplacian} and using Equation \ref{DeltaS} and Equation \ref{gradSi},  
\begin{align*}
    \tilde{\Delta}\log S \geq &\frac{1}{S}\left(\sum_{i,j} \lambda_i^{-1}\lambda_j^{-1}\abs{\tilde{g}_{i\bar{j},j}}^2-CRS-CR\right) \\&-\frac{1}{S^2}\left(S\sum_{i,j} \lambda_i^{-1}\lambda_j^{-1}|\tilde g_{i\bar j,j}|^2 +CR -S^2\gamma''(u)\abs{\partial u}^2_{\tilde{\omega}_{\varepsilon}}-\frac{\gamma^\prime\pare{u}^2}{\gamma^{\prime\prime}\pare{u}}\cdot CR\right)\\
    =& -CR-\frac{CR}{S}-\frac{CR}{S^2}\pare{1+\frac{\gamma^\prime\pare{u}^2}{-\gamma^{\prime\prime}\pare{u}}} + \gamma''(u)|\partial u|^2_{\tilde{\omega}_\varepsilon}.
\end{align*}
From the definition of $u$,
\begin{align*}
    \pot u = \tilde{\omega}_\epsilon -\pare{\theta+\pot\rho}-\epsilon \eta.
\end{align*}
Taking the trace with respect to $\tilde{\omega}_\epsilon$, we find
\begin{align*}
    \tilde{\Delta}u \le n - \pare{\delta+\epsilon}R\le n-\delta R.
\end{align*}
Therefore,
\begin{equation}\label{DeltaHi}
    \tilde \Delta H \geq (\delta \gamma'(u)-C)R-n\gamma'(u)-\frac{CR}{S}-\frac{CR}{S^2}\pare{1+\frac{\gamma^\prime\pare{u}^2}{-\gamma^{\prime\prime}\pare{u}}} 
\end{equation}
Note in particular that the gradient terms cancel. 
If $\gamma'(u)=A+\frac{1}{u}$, we have the simple bound
\[A\le A+\frac{1}{u}\le A+1\]
as well as 
\[-\gamma^{\prime\prime}\pare{u}=\frac{1}{u^2}.\] 
Hence, at the point $x_0$ which maximizes $H$, we can rewrite Equation \ref{DeltaHi} as 
\begin{align*}
    0 \ge \pare{\delta\gamma^\prime\pare{u}-C}R-n\gamma^\prime\pare{u}-\frac{CR}{S} -\frac{CR}{S^2}\cdot C_A\pare{1+u^2},
\end{align*}
where $C_A$ is a constant depending on $A$. Choosing $A$ large enough to ensure that 
\[\delta A -C\ge 1\] 
and absorbing fixed constants, we arrive at the following bound: 
\begin{equation}\label{RSBound}
    0 \ge R-C\pare{1+\frac{R}{S}+\frac{R}{S^2}\pare{1+u^2}}.
\end{equation}
\textit{Step 3. (Contradiction step)} Assume for the sake of contradiction that $H$ is not uniformly bounded, that is, there is a sequence $\epsilon_\ell\searrow 0$ such that the corresponding maximum points $x_\ell$ satisfy $H\pare{x_\ell}\nearrow\infty$. By the definition of $H$, 
\begin{align*}
    \log \frac{S\pare{x_\ell}}{u\pare{x_\ell}} &= H\pare{x_\ell}+Au\pare{x_\ell} \nearrow \infty.
\end{align*}
Therefore $\frac{S\pare{x_\ell}}{u\pare{x_\ell}}\nearrow\infty$ and since $u\ge 1$,  $S\pare{x_\ell}\nearrow \infty$. But, for sufficiently large $\ell$, Equation \ref{RSBound} implies $R\pare{x_\ell}\le 2C$. But,
\begin{align*}
    S\pare{x_\ell} = D\sum_i\prod_{j\ne i}\lambda_j^{-1}\le nDR\pare{x_\ell}^{n-1} \le nd_2\pare{2C}^{n-1}.
\end{align*}
Hence, $S$ is uniformly bounded, which is a contradiction.
\end{proof}

Fix $K\Subset U\Subset X\setminus E$. Apply Proposition \ref{OmegaEpsilonBounded} to $\overline{U}$ and consider
\begin{align*}
    \log\det\pare{b^\epsilon_{i\bar{j}}+\pare{\varphi_\epsilon}_{i\bar{j}}} = \log c_\epsilon + \log \Omega_{\text{loc}}.
\end{align*}
All derivatives of the right side have uniform bounds, as do all derivatives of $b^\epsilon_{i\bar{j}}$. Applying a local complex Evans-Krylov-Trudinger estimate and interior Schauder estimates (see Section 4 of \cite{TW-AJM} and Section 4 of \cite{GL23}), for all $k\ge 0$ and $0<\alpha<1$ 
\begin{equation}\label{PhiRegularity}
    \sup_{0<\epsilon<1}\norm{\varphi_\epsilon}_{C^{k,\alpha}\pare{K,\eta}}\le C_{K,k,\alpha}.
\end{equation}

To finish the proof of Theorem \ref{MainTheorem1}, we now deal with the limiting behavior.

\begin{prop}\label{EverythingConverges}
    Let $\epsilon_j\searrow 0$. After passing to a subsequence $\varepsilon_{j_k}$ there is a bounded function $\varphi_0\in \mathrm{PSH}(X,\theta)$ such that 
    \begin{align*}
        c_{\epsilon_{j_k}} &\longrightarrow c_0 \\
        \varphi_{\epsilon_{j_k}} &\longrightarrow \varphi_0
    \end{align*}
    where the convergence for the second equation is taken in $L^1\pare{X}\cap C^\infty_{\text{loc}}\pare{X\setminus E}$. Moreover,
    \begin{equation}\label{CMALimited}
        \sup_X\varphi_0 = 0, \quad \inner{\pare{\theta+\pot \varphi_0}^n} = c_0\Omega.
    \end{equation}
    If $\omega_Y$ is Kähler, then $c_0=\operatorname{vol}\pare{\theta}$.
\end{prop}
\begin{proof}
    Note that $\varphi_{\epsilon_j}\in \mathrm{PSH}(X,\theta_1)$ and furthermore by Proposition \ref{ZerothOrder} 
    \begin{align*}
        -C\le \varphi_{\epsilon_j}\le 0, \quad \sup_X\varphi_{\varepsilon_j} = 0.
    \end{align*}
    By compactness of normalized quasi-plurisubharmonic functions, there exists an $L^1$-convergent subsequence $\varphi_{\varepsilon_{j_k}}$. Additionally, Hartogs compactness shows that the upper semicontinuous representative of the limit satisfies $\sup_X\varphi_0 = 0$. In the sense of currents, we thus have
    \begin{align*}
        \theta_{\epsilon_{j_k}}+\pot \varphi_{\epsilon_{j_k}} \rightharpoonup  \theta + \pot \varphi_0\ge 0.
    \end{align*}
    The uniform pointwise bounds pass to the limit, so $\varphi_0$ is similarly bounded. Since $c_{\epsilon_j}\in\sqbr{\MAlow\pare{\theta},C}$ by Proposition \ref{Boundedcepsilon}, after passing to a further subsequence, we see that $c_\epsilon\to c_0$.

    In fact, we can say something stronger about $c_0$. Firstly, every $c_{\epsilon_j}\to c_0$. If $\epsilon_{i}<\epsilon_j$, then evaluating at the maximum of $\varphi_{\epsilon_i}-\varphi_{\epsilon_j}$, we have 
    $\tilde{\omega}_{\epsilon_i}<\tilde{\omega}_{\epsilon_j}$
    and hence the constants obey $c_{\epsilon_i}<c_{\epsilon_j}$, i.e. $c_\epsilon$ is strictly increasing in $\epsilon$. Thus, every sequence $\varepsilon_j\searrow0$ corresponds to a sequence $c_{\varepsilon_j}\to c_0$. Note that we do still require the choice of subsequence to establish $\varphi_0$. Secondly, while $c_0$ need not have a cohomological interpretation, it is independent of the auxiliary Hermitian metric. For hermitian metrics $\eta_1,\eta_2$ choose $M\ge 1$ so that $M^{-1}\eta_1\le \eta_2\le M\eta_1$. The same maximum comparison gives for $\epsilon>0$ sufficiently small $c_{\epsilon/M}^{\pare{1}}\le c_\epsilon^{\pare{2}}\le c_{M\epsilon}^{(1)}$. Passing to the limit gives the independence. 

    By Equation \ref{PhiRegularity} and a diagonal argument, we have smooth convergence on compact subsets of $X\setminus E$. Passing to the limit in the smooth equation, we see
    \begin{align*}
        \inner{\pare{\theta+\pot\varphi_0}^n} = c_0\Omega \quad\text{on }X\setminus E.
    \end{align*}
    The defining property of the non-pluripolar product gives that the left-hand side does not charge $E$; see \cite{Li} Section 2.1 and \cite{BGL}. Since $\Omega$ is smooth and $E$ is a proper analytic subset, $\Omega\pare{E}=0$. Equality away from $E$, together with zero mass on $E$, implies global equality of the measures.

    Now we only need to prove that $c_0=\operatorname{vol}\pare{\theta}$ in the K\"ahler setting. Bedford-Taylor integration by parts and Stokes' theorem yield
    \begin{equation*}
        c_0=\int_X c_0\Omega = \int_X \pare{\theta+\pot \varphi_0}^n =\int_X \theta^n = \operatorname{vol}\pare{\theta}.\qedhere
    \end{equation*}
\end{proof}

It remains to identify the limiting potential in the \Kahler case. For the remainder of the paper we assume that $\theta$ is closed.

\begin{lem}\label{UniquePullback}
If $\theta=\pi^*\omega_Y$, then every bounded $u\in \mathrm{PSH}(X,\theta)$ is of the form $\pi^*\bar{u}$, for a unique bounded $\bar{u}\in \mathrm{PSH}(Y,\omega_Y)$.
\end{lem}
\begin{proof}
The fibers of $\pi$ are connected by Zariski's main theorem, and $\theta$ vanishes along any reduced fiber. The restriction of $u$ to any irreducible component is a bounded plurisubharmonic function on a compact complex space and thus constant. Connectedness forces the constants to agree, so $u$ is constant on every fiber and $\bar{u}\pare{y}=u\pare{x}$ for $x\in\pi^{-1}\pare{y}$ is thus a well-defined bounded upper-semicontinuous function.

To show positivity, let $h$ be a local potential for $\omega_Y$ on an open set $U$. Then the function $u+\pi^*h=\pi^*\pare{\bar{u}+h}$ is plurisubharmonic on $\pi^{-1}\pare{U}$. The descent criterion for plurisubharmonic functions under a proper modification of a normal complex space gives that $\bar{u}+h$ is plurisubharmonic, and thus $\bar{u}\in \operatorname{PSH}\pare{Y,\omega_Y}$; see {\cite[Theorem~6.3]{EGZ}} for further details.
\end{proof}

\begin{proof}[Proof of Corollary \ref{KahlerMainTheorem1}]
    By compactness and {\cite[Lemma~6.4]{EGZ}}, there are $\delta>0$ and $f\in L^{1+\delta}\pare{Y,\omega_Y^n}$ such that 
\begin{align*}
    \Omega_Y &= f\omega_Y^n.
\end{align*}
Then by {\cite[Theorem~6.3]{EGZ}}, there is a unique bounded normalized function $\psi_Y\in \mathrm{PSH}(Y,\omega_Y)$ such that 
\begin{align*}
    \pare{\omega_Y+\pot \psi_Y}^n = \operatorname{vol}\pare{\theta}\Omega_Y, \quad \sup_Y\psi_Y = 0.
\end{align*}
Its pullback is thus a bounded function $\varphi_0\in \mathrm{PSH}(X,\theta)$ solving Equation \ref{CMALimited}. 

Conversely, if $\varphi_0=\pi^*{\bar{\varphi}_0}$, we can push the equation forward to $Y_{\text{reg}}$. Since neither measure charges $Y_{\text{sing}}$,  
\begin{align*}
    \pare{\omega_Y+\pot {\bar{\varphi}_0}}^n = \operatorname{vol}\pare{\theta}\Omega_Y.
\end{align*}
Therefore, the uniqueness theorem from \cite{EGZ} and the normalization imply that $\varphi_0=\pi^*\psi_Y$. This finishes the proof of Theorem \ref{MainTheorem1}.
\end{proof}

The preceding argument also yields the following abstract closed-background result on $X$.

\begin{theorem}\label{MainTheorem1NoY}
    Let $X$ be a compact connected complex manifold of complex dimension $n$. Let $\theta$ be a smooth closed semipositive big $(1,1)$-form and $\eta$ a hermitian metric on $X$. Assume that $\MAup \pare{\theta+\eta}<\infty$. Let $\rho\in \mathrm{PSH}(X, \theta)$ be as in Equation \ref{thetaboundedbelow}, with analytic singularities along $E=\{z\in X: \rho(z) = -\infty\}$ and smooth on $X\setminus E$, and satisfying
    \[
    \theta+\sqrt{-1} \partial\overline\partial \rho\geq \delta \eta
    \]
    for some $\delta>0$. If $\Omega$ is a smooth positive probability volume form, the solutions
    \begin{align*}
        \pare{\theta+\epsilon\eta+\pot\varphi_\epsilon}^n=c_\epsilon\Omega,\quad\sup_X\varphi_\epsilon=0
    \end{align*}
    satisfy uniform upper and lower bounds for $c_\epsilon$ and a uniform $L^\infty$ estimate. Every sequence $\epsilon_j\searrow 0$ has a subsequence for which the potentials converge in $L^1(X)$, and smoothly on compact subsets of $X\setminus E$, to a bounded function $\varphi_0\in \mathrm{PSH}(X,\theta)$ solving 
    \begin{align*}
        \pare{\theta+\pot\varphi_0}^n=\operatorname{vol}\pare{\theta}\Omega
    \end{align*}
    in the non-pluripolar sense on $X$. If uniqueness holds in the relevant bounded pluripotential class, convergence is along the full family. Additionally, the limiting form $\tilde{\omega}_0=\theta+\pot\varphi_0$ is a smooth Kähler metric on $X\setminus E$ satisfying $\operatorname{Ric}\pare{\tilde{\omega}_0}=-\pot\log\Omega_{\text{loc}}$ on $X\setminus E$ and each approximating metric satisfies $\operatorname{Ric}^{\text{Ch}}\pare{\theta+\epsilon\eta+\pot\varphi_\epsilon}=-\pot\log\Omega_{\text{loc}}$.
\end{theorem}
\begin{proof}
    The only argument that has not yet been resolved concerns the general Chern-Ricci curvature. Since $\Omega$ and $\eta$ are fixed and smooth on the compact manifold $X$, there is a uniform constant $C_\Omega>0$ such that 
    \[\abs{\operatorname{Tr}_\eta-\pot\log\Omega_{\text{loc}}}\le C_\Omega.\] 
    The estimate for $(I)$ in Equation \ref{estimateeye} now becomes
    \begin{align*}
        (I) &\ge \sum_{i,j,k}\lambda_i^{-1}\lambda_j^{-1}\abs{\tilde{g}_{i\bar{j},k}}^2-CR-C_\Omega.
    \end{align*}
    Examining the $\tilde{\Delta}\log S$ lower bound, since $SR\ge n^2$ by Cauchy-Schwarz, we see
    \[-\frac{C_\Omega}{S}\ge -\frac{C_\Omega}{n^2}R.\] 
    The additional term $-C_\Omega/S$ is therefore bounded below by $-C_\Omega R/n^2$ and can be absorbed into the existing $CR$ term. The remaining estimates and compactness argument, therefore, apply unchanged. Finally, the asserted Ricci identities follow by applying $-\pot\log$ to both sides of the Monge-Ampère equations, using that their normalization constants are spatially constant.
\end{proof}

\section{Consequences on Conifolds}

In this section, we focus on the case when $Y$ is a projective threefold with ordinary double point singularities, and $\pi\colon X\to Y$ is a small crepant resolution with $E=\mathrm{Exc}(\pi)$. We denote an exceptional $(-1,-1)$-curve on $X$ by $C$. We denote a  \Kahler metric on $Y$ by $\omega_Y$, and fix a hermitian form $\eta$ on $X$ which is \Kahler in a small neighborhood of $E$ thanks to Lemma \ref{EtaKahlerNearE}. The forms
\[
\tilde{\omega}_\varepsilon = \pi^*\omega_Y + \varepsilon\eta+\pot\varphi_\varepsilon 
\]
are \Kahler and Ricci-flat on the same neighborhood. 
The calculations in \cite{Song} are local and apply without assuming $X$ is globally Kähler.

Identify a neighborhood of $C$ with a relatively compact neighborhood of the zero section in the local model $\hat{E}$ of Equation \ref{NodalLocalModel}. 
In bundle coordinates $\pare{z,\xi_1,\xi_2}$, set
\begin{align*}
    e^\rho &= \pare{1+\abs{z}^2}\pare{\abs{\xi_1}^2+\abs{\xi_2}^2}\\
    \omega_{\hat{E}} &= \pot e^\rho\\
    \hat{\omega} &= \omega_{FS}+\omega_{\hat{E}}.
\end{align*}
After rescaling the coordinates so that $e^\rho<1$, set $\mathcal{U}=\brac{\rho<0}$.
By our calculations in Section 3,
\[
(\tilde{\omega}_\varepsilon)^n = c_\varepsilon\Omega
\]
where $\Omega$ is a smooth positive volume form on $X$. Consider a neighborhood $N$ of $\overline{\mathcal{U}}$, and consider a collar $\mathcal{A}\Subset \mathcal{N}\setminus C$ of $\partial\mathcal{U}$. Then there is a uniform constant $B>1$ such that on $\mathcal{A}$ the following collar estimate holds:
\begin{equation}\label{CollarEstimate}
    B^{-1}\hat{\omega} \le \tilde{\omega}_\epsilon \le B\hat{\omega}.
\end{equation}
The following proposition allows us to get control of $\tilde{\omega}_\varepsilon$ away from $C$. 

\begin{prop}
There is a constant $D>0$ independent of $\epsilon$ such that 
\begin{equation}\label{omegaboundedoffcurve}
    D^{-1}\omega_{\hat{E}} \le \tilde{\omega}_\epsilon \le De^{-\rho}\omega_{\hat{E}}\quad\text{on }\mathcal{U}\setminus C.
\end{equation}
\end{prop}
\begin{proof}
For the lower bound, we follow a similar argument as in \cite{YauSchwarz}. Locally define
\begin{align*}
    q &= \operatorname{Tr}_{\tilde{\omega}_\epsilon}(\omega_{\hat{E}}).
\end{align*}
More geometrically, note that we can embed $\psi\colon \mathfrak{C}\to \mathbb{C}^4$, hence we can consider the composition $F=\psi\circ \pi\colon \mathcal{U}\to \mathbb{C}^4$. In particular, $\omega_{\hat{E}}=F^*\omega_{\mathbb{C}^4}$, i.e. the pullback of a flat metric. 
Since $\tilde{\omega}_\epsilon$ is Ricci-flat, the Chern-Lu inequality applied to $F$ gives 
\begin{align*}
    \Delta_{\tilde{\omega}_\epsilon} \log q \ge 0.
\end{align*}
Since $q$ extends smoothly across $C$ and $dF$ has rank two along $C$, $q>0$ there. Thus $\log q$ is well-defined and subharmonic. The collar estimate in Equation \ref{CollarEstimate} and the maximum principle imply that $q\le 3 B$ on $\mathcal{U}$ and thus if $D_1=3B$, we have
\begin{align*}
    {D_1}^{-1}\omega_{\hat E}\le \tilde{\omega}_\varepsilon.
\end{align*}
As for the upper bound, by unitary rotation of the fibers and a change of basis, we may calculate volumes at $z=0$ and $\pare{\xi_1,\xi_2}=\pare{r,0}$. In the coordinate frame $\pare{dz,d\xi_1,d\xi_2}$, the hermitian matrices of $\omega_{\hat{E}}$ and $\hat{\omega}$ are, up to fixed normalization,
\begin{align*}
    \operatorname{diag}\pare{r^2,1,1}\quad\text{and }\quad \operatorname{diag}\pare{1+r^2,1,1}.
\end{align*}
Since $e^\rho=r^2$ at this point and  $r\le1$, this proves
\begin{align*}
    \hat{\omega}^3\le Ke^{-\rho}\omega_{\hat{E}}^3.
\end{align*}
On $\overline{\mathcal{U}}$ the smooth volume $\Omega$ is uniformly comparable to $\hat{\omega}^3$, thus
\begin{align*}
    \tilde{\omega}_\epsilon^3=c_\varepsilon\Omega \le C\hat{\omega}^3\le D_2e^{-\rho}\omega_{\hat{E}}^3.
\end{align*}
If $\lambda_1,\lambda_2,\lambda_3$ are the eigenvalues of $\tilde{\omega}_{\epsilon}$ relative to $\omega_{\hat{E}}$, then the lower metric bound gives $\lambda_j\ge D_1^{-1}$, but the preceding volume comparison gives $\lambda_1\lambda_2\lambda_3\le D_2e^{-\rho}$. Thus, every $\lambda_j\le D_1^2D_2e^{-\rho}$, proving the upper bound. To finish, take $D=\max\{D_1, D_1^2D_2\}$. 
\end{proof}

The upper estimate in Equation \ref{omegaboundedoffcurve} degenerates along $C$, so a separate restriction argument is required. We work under the same change of coordinates and normalization as in the proof of the previous proposition. We see that 
\begin{align*}
    e^{-\rho}\frac{\omega_{\hat{E}}^3}{\omega_{\text{FS}}\wedge\omega_{\hat{E}}^2}
\end{align*}
extends to a bounded function across the zero section. On $\mathcal{U}\setminus C$,  
\begin{align*}
    \frac{\tilde{\omega}_{\epsilon}\wedge \omega_{\hat{E}}^2}{\omega_{\text{FS}}\wedge\omega_{\hat{E}}^2} \le De^{-\rho}\frac{\omega_{\hat{E}}^3}{\omega_{\text{FS}}\wedge\omega_{\hat{E}}^2}\le C_0.
\end{align*}
Notably, this bound holds near $C$, so taking the limit to $C$, we find 
\begin{align*}
    \tilde{\omega}_{\epsilon}\big|_C\le C_0\omega_{\text{FS}}.
\end{align*}
Additionally, because of Stokes' theorem and the fact that $\theta\big|_C=0$, 
 \begin{align*}
    \int_C\tilde{\omega}_\epsilon=\epsilon\int_C\eta.
\end{align*}

We would like to turn this into a diameter bound along $C$.
\begin{lem}
Let $h$ be a hermitian metric on $\Pone$ satisfying
\begin{align*}
    h\le C_0 h_{\text{FS}},\quad \operatorname{Area}_h\pare{\Pone}\le A.
\end{align*}
If $0<A\le 1$, then 
\begin{align*}
    \operatorname{diam}\pare{\Pone,h}\le CA^{1/3}
\end{align*}
where $C$ is a constant depending only on $C_0$ and $h_{\text{FS}}$.
\end{lem}
\begin{proof}
Choose finitely many pairs of coordinate disks $U_j\Subset U_j^\prime$ such that the $U_j$ cover $\Pone$ and $h_{\text{FS}}$ is uniformly comparable to the Euclidean metric on $U_j^\prime$. Choose $A_0>0$ so that $A_0^{1/3}$ is smaller than the Lebesgue number. If $A\ge A_0$, then 
\begin{align*}
    \operatorname{diam}\pare{\Pone,h}\le \sqrt{C_0}\operatorname{diam}\pare{\Pone,h_{\text{FS}}}\le CA^{1/3},
\end{align*}
and we are done. So, assume $A<A_0$. Writing $h=fh_{\text{FS}}$, we see that $0\le f\le C_0$. Take two points in a shared coordinate disk, and after a change of coordinates, let their real coordinates be $\pare{0,0}$ and $\pare{x_0,0}$, with $x_0\le C$. By Fubini's theorem, there is $y^\prime\in \pare{-A^{1/3},A^{1/3}}$ such that
\begin{align*}
    \int_0^{x_0} f\pare{x,y^\prime}\, dx\le CA^{2/3}.
\end{align*}
So along the path from $(0,0)\to (0,y^\prime)\to\pare{x_0,y^\prime}\to\pare{x_0,0}$, by Cauchy-Schwarz the $h$-length of the horizontal segment is at most $CA^{1/3}$ and each of the vertical segments has $h$-length at most $CA^{1/3}$ as well. For arbitrary points, connect them through a uniformly bounded chain of overlapping coordinate disks from the fixed finite cover to get the desired diameter bound.
\end{proof}

Applying this lemma to our situation, we see that 
\begin{align*}
    \tilde{\omega}_\epsilon\big|_C \le C_0\omega_{\text{FS}},\quad \operatorname{diam}\pare{C,g_\epsilon|_C}\le C_0\epsilon^{1/3}
\end{align*}
for some uniform constant $C_0$. We would like to get a similar bound across all of $\mathcal{U}$, but this requires more machinery.

Define 
\begin{align*}
    S_1 = \brac{\xi_1=0},
\end{align*}
and on $\mathcal{U}\setminus S_1$ use the coordinates
\begin{align*}
    \nu_1=z\xi_1,\quad \nu_2=\xi_1,\quad w=\frac{\xi_2}{\xi_1}.
\end{align*}
The surfaces $L_w=\brac{w=\text{constant}}$ form a standard ruling. Set 
\begin{equation*}
    e^{\rho_1}=\pare{1+\abs{z}^2}\abs{\xi_1}^2=\abs{\nu_1}^2+\abs{\nu_2}^2,\quad \tau=\pot e^{\rho_1}.
\end{equation*}
On every $L_w$,
\begin{align*}
    \tau\big|_{L_w} &= \sqrt{-1}\pare{d\nu_1\wedge d\bar{\nu}_1+d\nu_2\wedge d\bar{\nu}_2}
\end{align*}
is flat. Define
\begin{equation*}
    H = \operatorname{Tr}_{\tau|_{L_w}}\pare{\tilde{\omega}_\epsilon\big|_{L_w}}.
\end{equation*}

\begin{lem}
On $\mathcal{U}\setminus S_1$,
\begin{align*}
    \Delta_{\tilde{\omega}_\epsilon}\log H\ge 0.
\end{align*}
\end{lem}
\begin{proof}
Let $p\in \mathcal{U}\setminus S_1$. Following Step 1 of the proof of Proposition 4.1 in \cite{Song}, choose holomorphic coordinates $\pare{x_1,x_2,x_3}$ preserving the slice condition $x_3=w$. Under this choice, the coefficient matrix of $\tilde{\omega}_\varepsilon$ is diagonalized at $p$, i.e.
\[
(\tilde g_{i\bar{j}})(p) = \mathrm{diag}(\lambda_1,\lambda_2,\lambda_3),
\]
while
\begin{equation*}
\tau\big|_{L_w}=\sqrt{-1}\pare{dx_1\wedge d\bar x_1+dx_2\wedge d\bar{x}_2}.
\end{equation*}
Hence, we find $H(p)=\lambda_1+\lambda_2$. Evaluating the following identities at $p$,  since $\tilde{\omega}_\varepsilon$ is Ricci-flat, we find
\begin{align*}
\sum_{k=1}^3\lambda_k^{-1}\tilde{g}_{k\bar{k},\alpha\bar{\alpha}}=\sum_{k=1}^3\sum_{p=1}^3\lambda_k^{-1}\lambda_p^{-1}\abs{\tilde{g}_{k\bar{p},\alpha}}^2.
\end{align*}
Furthermore, since $\tilde{\omega}_\epsilon$ is K\"ahler,  
\begin{align*}
    \Delta_{\tilde{\omega}_\epsilon}H = \sum_{\alpha=1}^2\sum_{k=1}^3\sum_{p=1}^3\lambda_k^{-1}\lambda_p^{-1}\abs{\tilde{g}_{k\bar{p},\alpha}}^2.
\end{align*}
On the other hand, Cauchy-Schwarz implies
\begin{align*}
    \abs{\del H}^2_{\tilde{\omega}_\epsilon} = \sum_{k=1}^3\lambda_k^{-1}\abs{\sum_{\alpha=1}^2 \tilde{g}_{\alpha\bar{\alpha},k}}^2 \le \pare{\lambda_1+\lambda_2}\sum_{\alpha=1}^2\sum_{k=1}^3\sum_{p=1}^3\lambda_k^{-1}\lambda_p^{-1}\abs{\tilde{g}_{k\bar{p},\alpha}}^2=H\Delta_{\tilde{\omega}_\epsilon}H.
\end{align*}
Thus, by Equation \ref{laplacian} we arrive at
\begin{equation*}
    \Delta_{\tilde{\omega}_\epsilon}\log H = \frac{\Delta_{\tilde{\omega}_\epsilon} H}{H} -\frac{\abs{\del H}^2_{\tilde{\omega}_\epsilon}}{H^2}\ge 0.\qedhere
\end{equation*}
\end{proof}

\begin{cor}\label{tracebound}
    There is a constant \(C>0\), independent of \(\epsilon\), such that
    \[
    H\leq Ce^{-\rho_1}
    \]
    on $\mathcal{U}\setminus S_1$.
\end{cor}

\begin{proof}
    For $\delta>0$, define the quantity
\begin{equation*}
    I_\delta = \log H+\pare{1+\delta}\rho_1.
\end{equation*}
For fixed $\epsilon$, smooth equivalence of $\tilde{\omega}_\epsilon$ and $\hat{\omega}$ gives $H\le C_\epsilon e^{-\rho_1}$. Thus, $I_\delta\to-\infty$ at $S_1$. Its outer boundary values are uniform by Equation \ref{CollarEstimate}. Away from $S_1$, $\pot \rho_1 = \omega_{\text{FS}}\ge 0$, so $I_\delta$ is subharmonic, and combining this with the previous lemma, the maximum principle, and taking $\delta\searrow 0$,  
\begin{equation*}
    H\le Ce^{-\rho_1}
\end{equation*}
away from $S_1$, for some uniform constant $C$.
\end{proof}

Define the holomorphic vector field $V$ on $\mathcal{U}\setminus S_1$ by
\begin{equation*}
    V= V_1+V_2, \qquad \text{where }V_1 = \xi_1\dfrac{\del}{\del \xi_1},\quad V_2 = \xi_2\dfrac{\del}{\del \xi_2}
\end{equation*}
and the normalized vector field
\begin{equation*}
    \quad W=e^{-\rho/2}V.
\end{equation*}
In the coordinates $\pare{\nu_1,\nu_2,w}$, we see that 
\begin{align*}
    V_1 = \nu_1\dfrac{\del}{\del \nu_1}+ \nu_2\dfrac{\del}{\del \nu_2}- w\dfrac{\del}{\del w},
\end{align*}
which is not tangent to a general fixed slice. At points in $L_0$, one has $V_1=V$. The Bochner formula for a holomorphic vector field on a Ricci-flat \Kahler manifold yields
\[
\Delta _{\tilde{\omega}_\varepsilon}|V_1|^2_{\tilde{\omega}_\varepsilon} = |\nabla ^{1,0} V_1|_{\tilde{\omega}_\varepsilon}^2
\]
where $\nabla^{1,0}$ is the Levi-Civita connection. Hence, combining Cauchy-Schwarz and Equation \ref{laplacian}, we see
\begin{equation*}
\Delta_{\tilde{\omega}_\epsilon}\log\abs{V_1}^2_{\tilde{\omega}_\epsilon}\ge 0
\end{equation*}
away from its zero set.

Define
\begin{equation*}
    G_\delta = \log H+\log\abs{V_1}^2_{\tilde{\omega}_\epsilon}+\delta\rho_1.
\end{equation*}
Similar calculations to those in the proof of Corollary \ref{tracebound} show that, for each fixed $\epsilon$, $G_\delta\to-\infty$ at $S_1$ and the outer boundary values are uniform. A similar application of the maximum principle and taking $\delta\searrow 0$ gives 
\begin{equation*}
    \abs{V_1}^2_{\tilde{\omega}_\epsilon}H\le C
\end{equation*}
away from $S_1$, for some uniform constant $C$.

At any prescribed point, after a unitary change of the two fiber coordinates, we may assume that the point lies on $L_0$. The constants remain uniform because the group $U(2)$ is compact, and thus the definition of the trace gives
\begin{align*}
    \abs{V}^2_{\tilde{\omega}_\epsilon}\le e^\rho H.
\end{align*}
Multiplying this by our previous bound and taking square roots, we get 
\begin{align*}
    \abs{V}^2_{\tilde{\omega}_\epsilon}\le Ce^{\rho/2}.
\end{align*}
Equivalently,
\begin{equation*}
    \abs{W}^2_{\tilde{\omega}_\epsilon}\le Ce^{-\rho/2}.
\end{equation*}

Let \(g\) be a Riemannian metric on $X$. We define the intrinsic length of a path $\gamma\colon [0,1]\to X$ between points $p,q\in X$ via
\[
L_g(\gamma) = \int_0^1 \sqrt{g(\gamma'(s), \gamma'(s))}\, ds.
\]
The intrinsic distance between two points $p,q\in U$ for a subset $U\subseteq X$ is then given by
\[
d_{U,g}(p,q)=\inf_{\substack{\gamma\colon\sqbr{0,1}\to U\\\gamma \text{ piecewise }C^1\\\gamma(0)=p, \gamma(1)=q}} L_{g}(\gamma).
\]
Finally, the intrinsic diameter is defined to be
\[
\mathrm{diam}^{\mathrm{int}}(U,g) = \sup_{p,q\in U} d_{U,g}(p,q). 
\]

We are now ready to locally bound the intrinsic diameter on $\mathcal{U}$.

\begin{prop}\label{LocalDiameterBound}
    For
    \begin{align*}
        U_\delta &= \brac{e^\rho<\delta^2}\Subset\mathcal{U},
    \end{align*}
    one has the uniform bound
    \begin{equation*}
        \operatorname{diam}^{\mathrm{int}}\pare{U_\delta,g_\epsilon}\le C\delta^{1/2}+C\epsilon^{1/3}.
    \end{equation*}
    In particular,
    \begin{equation*}
        \lim_{\delta\searrow 0}\limsup_{\epsilon\searrow 0}\operatorname{diam}^{\mathrm{int}}\pare{U_\delta,g_\epsilon}=0.
    \end{equation*}
\end{prop}
\begin{proof}
Set $R=e^{\rho/2}$. Along a fiberwise radial ray, the real radial direction is, up to a uniform constant, the real part of $W$. Since $\abs{W}^2_{\tilde{\omega}_\epsilon}\le Ce^{-\rho/2}$ implies $\abs{W}_{\tilde{\omega}_\epsilon}\le CR^{-1/2}$, a radial segment from radius $\delta$ to the zero section has length at most
\begin{align*}
    C\int_0^\delta R^{-1/2}\, dR\le C\delta^{1/2}.
\end{align*}
Any two points in $U_\delta$ can be joined by two such paths and a path in $C$. Since $\operatorname{diam}\pare{C,g_\epsilon|_C}\le C_0\epsilon^{1/3}$, this gives the desired bound.
\end{proof}

Finally, we also recall the following result of Zhang:

\begin{prop}[{\cite[Theorem~1.4]{Zhang2025}}]\label{zhang}
    Let $(X,p)$ be a germ of an isolated log terminal algebraic singularity, and $C(X,p)$ and $W(X,p)$ be the $K$-polystable Fano cone and the $K$-semistable Fano cone, respectively, canonically associated to the two-step degeneration of the singularity. Suppose $C(X,p)$ is isomorphic to $W(X,p)$ as Fano cones, and $C(X,p)$ has a smooth link.  Then every singular Kähler-Einstein metric with non-positive Ricci curvature and with bounded potential is conical at $p$. In particular, there exists a Calabi-Yau cone metric $g_C$, a number $\lambda>0$, and an identification $\Phi\colon U_0\to U_p$ between neighborhoods of the cone vertex $o$ and the singular point $p$, taking $o$ to $p$ and restricting to a diffeomorphism on the punctured neighborhoods, such that, for every integer $k\ge 0$,
    \begin{align*}
        \abs{\nabla_{g_C}^k\pare{\Phi^*g-g_C}}_{g_C} &= O\pare{r^{\lambda-k}},
    \end{align*}
    where $r$ is the cone radius. After perhaps shrinking the neighborhoods, $\Phi^*g$ and $g_C$ are uniformly equivalent on the punctured neighborhood.
\end{prop}

\section{Aspects of Gromov-Hausdorff Convergence}

Throughout this section, let $y_a\in Y$ be an ordinary double point. The germ $\pare{Y,y_a}$ is analytically isomorphic to a quadric affine cone 
\begin{equation*}
    (\mathfrak{C},o) = \set{(z_1,z_2,z_3,z_4)\in \C^4}{z_1^2+z_2^2+z_3^2+z_4^2=0}
\end{equation*}
Consider the singular Ricci-flat K\"ahler metric on $Y_{\text{reg}}$ given by
\begin{align*}
    \omega_0 &=\omega_Y+\pot\psi_Y
\end{align*}
where $\psi_Y$ is provided by Corollary \ref{KahlerMainTheorem1}.
Let the metric associated to $\omega_0$ be denoted by $g_0$
on $Y_{\text{reg}}$. Finally, let $d_0$ denote the intrinsic Riemannian distance. We define the extrinsic diameter to be
\[
\mathrm{diam}_{d_0}(V) = \sup_{p,q\in V}d_0(p,q).
\]
In particular, we have the following relation between the diameters:
\[
\mathrm{diam}_{d_0}(V)\leq \mathrm{diam}^{\mathrm{int}}(V, g_0)
\]

\begin{dfn}\label{DefNoNeck}
We say that the metric $g_0$ satisfies the punctured no-neck condition if, for every $y_a\in Y_{\text{sing}}$, there are nested connected relatively compact open neighborhoods $V_{a,\delta}$ forming a fundamental neighborhood basis at $y_a$ with 
\begin{equation*}
    \overline{V_{a,\delta^\prime}}\Subset V_{a,\delta} \quad \pare{\delta^\prime<\delta}, \quad \bigcap_{\delta>0}\overline{V_{a,\delta}} = \brac{y_a},
\end{equation*}
such that the extrinsic diameter obeys
\begin{equation}\label{PuncturedNoNeck}
    \lim_{\delta\searrow 0}\operatorname{diam}_{d_0}\pare{V_{a,\delta}\setminus\brac{y_a}}=0.
\end{equation}
\end{dfn}

\begin{lem}\label{NoNeckMet} Under the standing assumptions: 
\begin{enumerate}
    \item[(1).] The metric $g_0$ is polynomially asymptotic at $y_a$ to the conifold metric $g_{\mathfrak{C}}$.
    \item[(2).] There is a neighborhood basis $V_{a,\delta}$ of $y_a$ such that
\begin{align*}
    \operatorname{diam}_{d_0}
    \pare{V_{a,\delta}\setminus\brac{y_a}}
    \le C\delta.
\end{align*}
In particular, $g_0$ satisfies the punctured no-neck condition at $y_a$.
\end{enumerate}  
\end{lem}

\begin{proof} We split the proof into two parts:
\begin{enumerate}
    \item[(1).] We begin by verifying the hypotheses of Zhang's result \cite{Zhang2025}. The ordinary double point $(Y,y_a)$ is an isolated log terminal singularity and is biholomorphic to the vertex germ of $\mathfrak{C}$. Equipped with its standard Ricci-flat Kähler cone metric
    \[
    \omega_{\mathfrak{C}} = \sqrt{-1}\partial\overline{\partial} \left(\sum_{i=1}^4 |z_i|^2\right)^{2/3}
    \]
    the cone $\mathfrak{C}$ is a Calabi-Yau cone with smooth link $T^{1,1}\cong S^2\times S^3$ as shown in \cite{HeinSun2017}. Let $\xi_\mathfrak{C}$ be the Reeb field of this metric. Zhang observes after Theorem 1.4 that, for a germ biholomorphic to a Calabi-Yau cone, his cone hypothesis is satisfied. Applied here,
    \[
    W(Y,y_a)\cong C(Y,y_a)\cong (\mathfrak{C},o)
    \] 
    as Fano cones, where $(\mathfrak{C},o)$ is equipped with $\xi_\mathfrak{C}$. 
    Furthermore, the limiting metric is Ricci-flat and has a bounded potential on $Y$ near the ordinary double point $y_a$. To see this, choose a local potential $h$ for the Kähler form $\omega_Y$ on an open neighborhood $U_a$ of $y_a$, so that
\begin{align*}
    \omega_Y &= \pot h.
\end{align*}
Then on $U_a\setminus \{y_a\}$,
\begin{equation*}
    \omega_0 = \pot\pare{h+\psi_Y}.
\end{equation*}
The function $h+\psi_Y$ is a bounded plurisubharmonic local potential for $\omega_0$. By Theorem \ref{MainTheorem1}, $\omega_0$ is smooth and positive on $U_a\setminus \{y_a\}$ and $\operatorname{Ric}\pare{\omega_0}= 0$.
The metric $\omega_0$ is a singular Kähler-Einstein metric with bounded potential. Thus, all the hypotheses of Proposition \ref{zhang} are satisfied. By Zhang's observation,  the asymptotic identification in the cone-germ case may be chosen as a biholomorphism $\Phi\colon(U_o,o)\to (U_a,y_a)$ between neighborhoods of the vertices. For some $\lambda>0$, on $U_o\setminus \{o\}$ 
\begin{equation*}
    \abs{\Phi^*g_0-g_{\mathfrak{C}}}_{g_{\mathfrak{C}}
    } = O\pare{r^{\lambda}}.
\end{equation*}
In particular, after shrinking $U_a$, there is a constant $C>0$ such that
\begin{equation}\label{conebound}
    C^{-1}g_{\mathfrak{C}} \le \Phi^* g_0 \le Cg_{\mathfrak{C}}.
\end{equation}
\item[(2).] We now verify Definition \ref{DefNoNeck}. Let 
\begin{equation*}
    V_{a,\delta}=\Phi\pare{\brac{0<r<\delta}}\cup\{y_a\}.
\end{equation*}
Since $\Phi$ is a biholomorphism of pointed neighborhoods, it is easily verified that the $V_{a,\delta}$ have all of the required properties of Definition \ref{DefNoNeck}, except for Equation \ref{PuncturedNoNeck}. For the last property, take two points $p_1=(r_1, \xi_1)$ and $p_2=\pare{r_2,\xi_2}$ on $\mathfrak{C}$, ordered so that $r_1\leq r_2$. In order to estimate the distance between $p_1$ and $p_2$, let $0<s<r_1$. We radially join $p_1,p_2$ to the sphere $\mathcal{S}=\brac{(r, \xi)\in \mathfrak{C}: r=s}$ and then join the resulting $p_1^*,p_2^*$ along the sphere. The resulting path gives
\begin{align*}
    d_{\mathfrak{C}}\pare{p_1,p_2} &\le \pare{r_1-s}+\pare{r_2-s}+s\cdot d_{T^{1,1}}\pare{\xi_1,\xi_2}.
\end{align*}
Since $T^{1,1}$ is compact, it has finite diameter, and by letting $s\searrow 0$, we find
\begin{align*}
    d_{\mathfrak{C}}\pare{p_1,p_2} &\le r_1+r_2< 2\delta.
\end{align*}
The metric comparison Equation \ref{conebound} therefore gives
\begin{align*}
    \operatorname{diam}_{d_0}\pare{V_{a,\delta}\setminus\brac{y_a}} &\le C\delta
\end{align*}
and taking $\delta\searrow 0$, we establish the no-neck condition.\qedhere
\end{enumerate}
\end{proof}

We now prove a compactness result for the metric completion of $(Y_{reg},d_0)$. 

\begin{lem}\label{CompletionHomeo}
Under the standing assumptions, the metric completion $(\tilde{Y},d_{\tilde Y})$ of $(Y_{reg},d_0)$ is compact. Moreover, for every $y_a$, there is a unique completion point $\tilde{y}_a$, and the identity on $Y_{\text{reg}}$ extends to a homeomorphism
\begin{equation*}
    \iota\colon Y\to \tilde{Y},\quad \iota\pare{y_a}=\tilde{y}_a.
\end{equation*}
\end{lem}
\begin{proof}
For fixed $\delta>0$, the complement of the union of the $V_{a,\delta}$ is compactly contained in $Y_{\text{reg}}$, and thus is totally bounded for $d_0$. The sets $V_{a,\delta}\setminus\brac{y_a}$ can have arbitrarily small $d_0$-diameter, so $Y_{\text{reg}}$ is totally bounded, and its completion is consequently compact.

Choose $x_{a,k}\in V_{a,\delta_k}\setminus\brac{y_a}$ for a sequence $\delta_k\searrow 0$. Then $\pare{x_{a,k}}$ is Cauchy and its class is independent of all choices. Call this class $\tilde{y}_a$. The punctured no-neck condition also gives us that any sequence converging to $y_a$ in the topology of $Y$ is converging metrically to $\tilde{y}_a$.

{After choosing pairwise disjoint neighborhoods for two distinct nodes, every path joining the two neighborhoods must cross a fixed compact regular core. On this core, $g_0$ is uniformly positive, so the boundary components of the neighborhoods have positive mutual distance}. This means that distinct nodes give distinct completion points.

Define $\iota$ to be as above. Every Cauchy sequence has a convergent subsequence in the compact space $Y$, and our work above shows that $\iota$ is injective and surjective. Because $g_0$ is smooth, $\iota$ is continuous on $Y_{\text{reg}}$, and the continuity at the nodes follows from Equation \ref{PuncturedNoNeck}. Because $Y$ is compact and the metric completion is Hausdorff, $\iota$ is a homeomorphism.
\end{proof}

To finish, we prove the following more general result:

\begin{prop}\label{GHConverge}
Let $\pi\colon X\to Y$ be a continuous surjection between compact spaces $X$ and $Y$, whose restriction $X\setminus E\to Y\setminus S$ is a diffeomorphism, where $S$ is a finite nonempty set and $E=\pi^{-1}\pare{S}$. Let $g_j$ be a family of Riemannian metrics on $X$ and let $g_0$ be a smooth metric on $Y\setminus S$. Suppose the following conditions hold:
\begin{itemize}
    \item[(i)] $g_j\to \pi^*g_0$ smoothly on compact subsets of $X\setminus E$;
    \item[(ii)] the identity on $Y\setminus S$ extends to a homeomorphism onto the metric completion $\pare{\tilde{Y},d_{\tilde{Y}}}$ of $\pare{Y\setminus S,d_0}$
    \begin{align*}
        \iota\colon Y\to \pare{\tilde{Y},d_{\tilde{Y}}};
    \end{align*}
    \item[(iii)] for every $s_a\in S$, there are nested neighborhoods $V_{a,\delta}$ and $U_{a,\delta}=\pi^{-1}\pare{V_{a,\delta}}$ such that
    \begin{align*}
        \lim_{\delta\searrow 0}\operatorname{diam}_{d_{\tilde{Y}}}\pare{\iota\pare{V_{a,\delta}}}=0
    \end{align*}
    and 
    \begin{align*}
        \lim_{\delta\searrow 0}\limsup_{j\nearrow\infty}\operatorname{diam}^{\mathrm{int}}\pare{U_{a,\delta},g_j}=0.
    \end{align*}
\end{itemize}
Then $\pare{X,g_j}$ converges in Gromov-Hausdorff topology to the completion $(\tilde{Y},d_{\tilde{Y}})$.
\end{prop}
\begin{proof}
We prove convergence by showing that the graph of \(\iota\circ\pi\) is a correspondence whose distortion tends to zero. Fix $\tau>0$, let $|S|=N$, and choose neighborhoods of each $s_a\in S$ as in (iii) satisfying 
\begin{equation}\label{Ybound}
\mathrm{diam}_{\tilde Y}(\iota(V_{a,\delta}))<\tau
\end{equation}
and, for sufficiently large $j$,
\begin{equation}\label{Xbound}
\mathrm{diam}^{\mathrm{int}}(U_{a,\delta},g_j)<2\tau    
\end{equation}
and such that all neighborhoods are pairwise disjoint as $a$ varies. Furthermore, choose smaller neighborhoods $V_{a,\delta}^\prime\Subset V_{a,\delta}$ and $U_{a,\delta}^\prime=\pi^{-1}\pare{V_{a,\delta}^\prime}\Subset U_{a,\delta}$ such that $s_a\in V'_{a,\delta}$ and $\pi^{-1}(s_a)\subseteq  U_{a,\delta}'$. Define  
\[ 
\mathrm{Core}(X) = X\setminus \Big(\bigcup_a U_{a,\delta}'\Big) 
\] 
and similarly for $\mathrm{Core}(Y)$. Identify via $\pi$
\[ 
\pi|_{\mathrm{Core}(X)}\colon\mathrm{Core}(X)\to \mathrm{Core}(Y).\]
On $\mathrm{Core}(X)$, assumption (i) implies   \begin{align*}
\pare{1-o_j\pare{1}}\pi^*{g_0}\le g_j\le \pare{1+o_j\pare{1}}\pi^*{g_0}. 
\end{align*}
Thus, the lengths of curves in the cores are comparable by factors uniformly tending to $1$.
Consider an almost-minimizing path $\gamma\colon[0,1]\to Y\setminus S$ with endpoints $y_1, y_2\in \mathrm{Core}(Y)$ and corresponding endpoints $\pi^{-1}(y_1)=x_1$ and $\pi^{-1}(y_2)=x_2$ in $X$. Near each point $s_a$, if $\gamma$ enters $V_{a,\delta}'$ denote the first entrance to and last exit from $V_{a,\delta}'$ by $y_{a,1}^*$ and $y_{a,2}^*$ respectively, which are at most a distance of $\tau$ apart by Equation \ref{Ybound}, and discard the portion of $\gamma$ between $y_{a,1}^*$ and $y_{a,2}^*$. Repeat this process for the remaining path and each neighborhood $V_{a,\delta}'$ encountered. At most $N$ portions are discarded, with the remaining paths all lying in $\operatorname{Core}\pare{Y}$. Lifting the retained paths, we can join the endpoints $\pi^{-1}(y_{a,1}^*)=x_{a,1}^*$ and $ \pi^{-1}(y_{a,2}^*)=x_{a,2}^*$ within $U_{a,\delta}$ which are at most a distance of $2\tau$ apart by Equation \ref{Xbound}. If $a_j\to 1$ denotes the length-comparison factor between the metrics $g_j$ and $\pi^*g_0$ on the cores, letting the almost-minimizing errors tend to zero, by the triangle inequality
\begin{equation}\label{gjbound}
    d_{g_j}\pare{x_{1},x_{2}} \le a_j d_{\tilde{Y}}\pare{y_{1},y_{2}}+2N\tau.
\end{equation}
Conversely, we can begin with a $g_j$-almost-minimizing path in $X$ between two points $x_{1},x_{2}\in \mathrm{Core}(X)$ and perform a similar procedure on each $U_{a,\delta}'$ and project the retained paths to $Y$ via $\pi$. Applying the triangle inequality in this direction,
\begin{equation}\label{completionbound}
d_{\tilde{Y}}\pare{y_{1},y_{2}} \le a_j d_{g_j}\pare{x_{1},x_{2}}+N\tau.    
\end{equation}
See Figure \ref{fig:5.1proof} for a schematic of the argument.
\begin{figure}[H]
    \centering
    \includegraphics[width=0.75\linewidth]{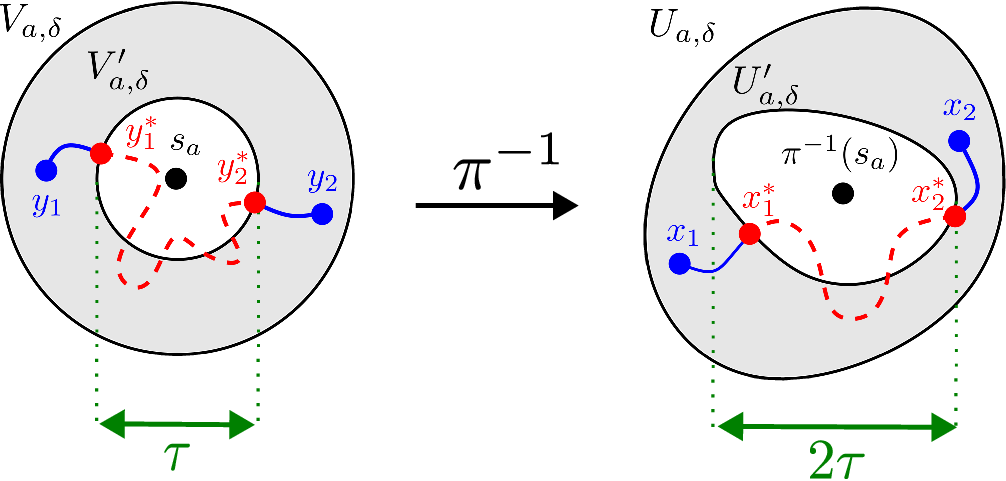}
    \caption{Schematic of the argument in the proof of Proposition \ref{GHConverge} locally near each $s_a$. On the left-hand side, the blue path is the part of the path which remains, while the red dashed path between the first and last intersection with $\partial U_{a,\delta}'$ is discarded. On the right-hand side is the image under $\pi$. The diameter bounds are given in green below.}
    \label{fig:5.1proof}
\end{figure}

We can now bound the $g_j$ distance between arbitrary points in $X$. For $i=1,2$, if $x_i\in \mathrm{Core}\pare{X}$, then set $p_i=x_i$. Otherwise, if $x_i\in U^\prime_{a,\delta}$, set $p_i$ to be a regular point in $U_{a,\delta}\cap \operatorname{Core}\pare{X}$. In either case, we see
\begin{equation*}
        d_{g_j}\pare{x_i,p_{i}} \le 2\tau.
    \end{equation*}
Set $G=\iota\circ\pi\colon X\to\tilde Y$, and define $y_i=G(x_i)$ and $q_i=G(p_i)$ for $i=1,2$. Then
    \begin{equation*}
        d_{\tilde{Y}}\pare{y_i,q_{i}} \le \tau.
    \end{equation*}
Applying Equation \ref{gjbound} and Equation \ref{completionbound} to $p_1$ and $p_2$ and the endpoint bounds gives
\begin{align}\label{diambounded}
d_{g_j}(x_1,x_2)
&\le a_jd_{\tilde Y}(G(x_1),G(x_2))
 +(2N+2a_j+4)\tau,\\
\label{diambounded2}d_{\tilde Y}(G(x_1),G(x_2))
&\le a_jd_{g_j}(x_1,x_2)
 +(N+4a_j+2)\tau.
\end{align}
Since the metric completion $(\tilde{Y}, d_{\tilde Y})$ is compact, by (ii) we have that 
$\mathrm{diam}(\tilde Y, d_{\tilde Y})=D<\infty$, so Equation \ref{diambounded} becomes
\begin{equation*}
    \mathrm{diam}(X,g_j)\leq a_jD+(2N+2a_j+4)\tau.
\end{equation*}
In particular, the diameter bound is uniform in $j$. Hence, by Equation \ref{diambounded} and Equation \ref{diambounded2}
\begin{align*}
d_{g_j}(x_1,x_2)-d_{\tilde Y}(G(x_1),G(x_2))
&\le (a_j-1)D+(2N+2a_j+4)\tau,\\
d_{\tilde Y}(G(x_1),G(x_2))-d_{g_j}(x_1,x_2)
&\le (a_j-1)\operatorname{diam}(X,g_j)
 +(N+4a_j+2)\tau.
\end{align*}
Since $\pi$ is surjective and $\iota$ is a homeomorphism by (ii), $G\colon X\to \tilde{Y}$ is surjective. Furthermore, $\mathrm{Graph}(G)$ is a correspondence between $X$ and $\tilde{Y}$, and its distortion satisfies
 \begin{equation*}
    \sup_{x_1,x_2\in X}\abs{d_{g_j}\pare{x_1,x_2}-d_{\tilde{Y}}\pare{G\pare{x_1},G\pare{x_2}}} \le CN\tau+o_j\pare{1}.
\end{equation*}
Taking $j\nearrow\infty$ and then $\tau\searrow 0$ finishes the proof.
\end{proof}

\begin{proof}[Proof of Theorem \ref{MainTheorem2}]
    Let $\epsilon_j\searrow 0$ and set $g_{j}=g_{\epsilon_j}$. By Theorem \ref{MainTheorem1}, Lemma \ref{NoNeckMet}, and Lemma \ref{CompletionHomeo} we meet conditions (i) and (ii) of Proposition \ref{GHConverge}. For each ordinary double point $y_a\in Y$, let $U_{a,\delta}$ be the saturated neighborhoods from Proposition \ref{LocalDiameterBound} and set $V_{a,\delta}=\pi\pare{U_{a,\delta}}$. These $V_{a,\delta}$ form a neighborhood basis satisfying $U_{a,\delta}=\pi^{-1}\pare{V_{a,\delta}}$. 

    Given $\tau>0$, Lemma \ref{NoNeckMet} provides an open neighborhood $W$ of $y_a$ such that $\operatorname{diam}_{d_0}\pare{W\setminus\brac{y_a}}<\tau$. For sufficiently small $\delta$, $V_{a,\delta}\subset W$, so $\operatorname{diam}_{d_0}\pare{V_{a,\delta}\setminus\brac{y_a}}<\tau$. The completion map $\iota$ preserves distances on $Y_{\text{reg}}$ and $\iota\pare{y_a}$ lies in the closure of $\iota\pare{V_{a,\delta}\setminus\brac{y_a}}$, so adjoining this point does not affect the diameter. Thus, $\lim_{\delta\searrow 0}\operatorname{diam}_{d_{\tilde{Y}}}\pare{\iota\pare{V_{a,\delta}}}=0$. Along with Proposition \ref{LocalDiameterBound}, these neighborhoods satisfy condition (iii) of Proposition \ref{GHConverge}. Thus, $\pare{X,g_j}$ converges in Gromov-Hausdorff topology to ${\pare{\tilde{Y},d_{\tilde{Y}}}}$. We see that this convergence is independent of the sequence of $\epsilon_j$. Hence, the full family converges.
\end{proof}

\end{document}